\documentclass[11pt,twoside, a4paper, english, reqno]{amsart}
\usepackage[dvips]{epsfig}
\usepackage{amscd}
\usepackage{amssymb}
\usepackage{amsthm}
\usepackage{amsmath}
\usepackage{latexsym}

\usepackage{upref}
\usepackage{hyperref}
\usepackage{color}
\theoremstyle{plain}
\newtheorem{thm}{Theorem}[section]
\theoremstyle{plain}
\newtheorem{lem}[thm]{Lemma}
\newtheorem{prop}[thm]{Proposition}

\theoremstyle{definition}
\newtheorem{defi}{Definition}[section]
\newtheorem{rem}{Remark}

\newcommand{\De} {\Delta}
\newcommand{\la} {\lambda}

\newcommand{\rn}{\mathbb{R}^{N}}

\numberwithin{equation}{section} \allowdisplaybreaks

     \title[]{ Sign changing solutions of the Hardy-Sobolev-Maz'ya equation}

\date{}

 \author[Debdip Ganguly]{Debdip Ganguly}
 \address[Debdip Ganguly]{\newline
  Centre for Applicable Mathematics,
  Tata Instiute of Fundamental Research,
   P.O.\ Box 6503, GKVK Post Office,
   Bangalore 560065, India}
 \email[]{debdip@math.tifrbng.res.in}

\begin{document}

\begin{abstract}
 In this article we will study the existence, multiplicity and Morse index of sign changing
 solutions for the Hardy-Sobolev-Maz'ya(HSM) 
 equation in bounded domain and  involving critical growth. We obtain infinitely many sign changing solutions for HSM equation. 
We also establish an estimate on the Morse index for the sign changing
solutions.
\end{abstract}

 \maketitle

\section{Introduction}
 In this article we will study the equation

\begin{equation}\label{E:1.1}
\left.
 \begin{array}{rlll}
  -\De  u-  \dfrac{\la u}{|y|^{2}} & = & \dfrac{|u|^{2^*(t)-2}u}{|y|^{t}} + \mu u &{\rm in} \; \Omega, \\
u & = &  0 &{\rm on} \; \partial \Omega,
 \end{array} \right\}
\end{equation}

 where $\Omega$ denotes a  bounded domain in $\rn
  \equiv \mathbb{R}^{k} \times \mathbb{R}^{N-k}$, $2 < k  < N$, $\mu >0$, $0\leq \la < \frac{(k-2)^{2}}{4},$
   $ 0 \leq t <2$ and  $2^{*}(t) = \frac{2(N-t)}{N-2}.$
 A point $x \in \rn$ is 
denoted as $x = (y,z) \in  \mathbb{R}^{k} \times \mathbb{R}^{N-k}$ and $\Omega$ contains some points $x^{0} = (0, z^0)$.  \\
 By a weak solution of the above problem we mean $u \in H^{1}_{0}(\Omega)$ satisfying 
\begin{equation}\label{WK}
\int_{\Omega} \left(\nabla u \nabla v - \la \frac{uv}{|y|^{2}} \right) dx = \int_{\Omega} \frac{|u|^{2^*(t)-2}uv}{|y|^{t}} dx 
+ \mu \int_{\Omega} uv \ dx,   \; \; \forall v\in H^{1}_{0}(\Omega). 
\end{equation}

  Eq. \eqref{E:1.1} has been used to model several astrophysical phenomenon like
stellar dynamics (see \cite{BFH}, \cite{BT}). Also, from the mathematical point of view, Eq. 
\eqref{E:1.1} with $\Omega= \rn$ has generated lot of interest  due to  its connection with  
the Brezis-Nirenberg problem in the Hyperbolic space
(see \cite{SMO}, \cite{HS}, \cite{MS}, \cite{SMF} ).\\

In recent years, much attention has been given to the existence of nontrivial solutions
 for the problem \eqref{E:1.1}. In a bounded domain, the problem \eqref{E:1.1} does not have a solution in general 
 due to the critical nature of the equation.
 For the case $\mu = 0$ and $2 \leq k < N$,  Bhakta and Sandeep in \cite{SB}, proved nonexistence of nontrivial solutions for the 
Eq. \eqref{E:1.1}, when $\Omega$ is star shaped with respect to the point $(0,z_{0}),$ 
using Pohozaev identity. They also discussed
existence in some special bounded domain. Jannelli in \cite{J}, has considered the problem \eqref{E:1.1} 
with $t = 0$ and $k = N$ and 
proved the existence of positive solution under some conditions on $\la$ and $\mu$. In \cite{CH}, Cao and Han 
established that  the Eq. \eqref{E:1.1}
 with $t = 0$ and $k = N $ admits a nontrivial solution for all  $\mu > 0$ if
 $\la \in [0, (\frac{(N-2)}{4})^2 - (\frac{N+2}{N})^2 ) .$\\

 When $\Omega = \rn$, the existence of positive solution for \eqref{E:1.1} has been studied in (\cite{RM}) and (\cite{TT}).
 Moreover, the  qualitative properties like cylindrical symmetry, regularity, decay 
properties and uniqueness of the positive solution of  Eq. \eqref{E:1.1} are thoroughly discussed 
 in (\cite {MG}) and (\cite{SMF}). Also when $\Omega = \rn$, the hyperbolic symmetry
of the equation  (see  \cite{HS}, \cite{MS}, \cite{SMF}) plays a crucial role in the study of non degeneracy of positive solutions.\\

 The Eq. $\eqref{E:1.1}$ with $\la = 0$, $t = 0$ is the well known Brezis-Nirenberg problem(\cite{BN}) 
 and is well studied(see \cite{GS} \cite{SD}, \cite{SZ} and references therein). 
 Cao and Yan in \cite{CY} considered problem \eqref{E:1.1} with $t = 0,$  $k = N$. They proved the existence 
of infinitely many solutions for any $\mu > 0$ if $\la \in [0 , (\frac{N-2}{2})^2 - 4)$ 
 and later Wang and Wang in \cite{WW} obtained the same result for \eqref{E:1.1} if  
$\la \in [0 , (\frac{k-2}{2})^2 - 4)$. All these results uses the compactness of
 the solutions of the Brezis-Nirenberg problem established
by Solimini and Devillanova \cite{SD} for $N \geq 7$. But \cite{CY} and  \cite{WW} do not have any information 
about the existence and multiplicity of sign changing solutions. It is also worth mentioning that, one cannot
obtain the existence and multiplicity of sign changing solutions of  problem \eqref{E:1.1}, by adopting the 
methods introduced in \cite{WW}.\\

So the question of existence of infinitely many sign changing solutions 
for the Eq. \eqref{E:1.1} remains open. An important result attributed to Schechter and Zou(see \cite{SZ}) asserts that there 
exists infinitely many sign changing solutions to the Brezis-Nirenberg problem in higher dimensions. Also Ganguly and Sandeep
 in \cite{DS}, proved the existence of infinitely 
many sign changing solutions for the Brezis-Nirenberg problem in the hyperbolic space. \\

In the literature, the only paper which deals with the existence of sign changing solutions for the Eq. \eqref{E:1.1} 
with $t = 0$ and $k =N$ 
is \cite{CP1}, where Cao and Peng obtained a pair of sign changing solutions for $N \geq 7$, $0 \leq\la < \frac{(N-2)^2}{4}-4$ and
 $ 0 < \mu < \mu_{1}(\la)$.\\

The novelty of this article is to obtain infinitely many sign changing solutions for the Eq. \eqref{E:1.1}. We establish an estimate
on Morse index of sign changing solutions (see Theorem \ref{T:3.1})  which led us  the following existence Theorem :

\begin{thm}\label{main}
 If $N > 6+t$ , $\mu > 0$, and $\la \in [0, \frac{(k-2)^{2}}{4} - 4 )$, then \eqref{E:1.1} has infinitely many sign changing 
solutions.
\end{thm}

\begin{rem}\label{main1}
  Recently, Chen and Zou in \cite{CZ}, proved \eqref{E:1.1} has infinitely many 
 sign changing solutions when  $k = N$, $t = 0$, $N \geq 7,$ $\mu > 0$ and $\la \in [0, \frac{(N-2)^{2}}{4} - 4 )$.
\end{rem}

\begin{rem}
Theorem \ref{main} is the cylindrical version of the result in \cite{CZ}. However 
the condition on $\lambda$ in Theorem \ref{main} is coming due to the compactness properties of solutions of Eq. 
\eqref{E:2.3} (see Lemma \ref{COMPACT}) which in turn gives $ k > 6.$
\end{rem}

As mentioned before, due to the critical nature of the Eq. \eqref{E:1.1}, the problem exhibits nonexistence phenomenon.
 First a definition:
\begin{defi}
 Let $\Omega$ be an open subset of $\rn$ with smooth boundary. We say that $\partial \Omega$ is orthogonal to the singular set if 
 for every $(0,z_{0}) \in \partial \Omega$ the normal at $(0, z_{0})$ is in $\{0\} \times \mathbb{R}^{N -k}$.
\end{defi}

 We prove the following  nonexistence result:

\begin{thm}\label{non}
 When $\mu \leq 0$, the Eq. \eqref{E:1.1} does not admits a non-trivial solution if $\Omega
 $ is star shaped with respect with respect to some point $(0, z_{0})$ and $\partial \Omega$ 
 is orthogonal to singular set.
\end{thm}

\begin{rem}
 When $\la = t = 0$, the Eq. \eqref{E:1.1} is well studied in \cite{SZ}. Hence we assume  either $\la > 0$ or 
$t > 0$. 
\end{rem}
   
 We divide the article in to four sections. Section 2  dicusses the notations and preliminaries, Section 3 is devoted
 to the existence and the estimate of Morse index of sign changing solutions.
 The results of Section 3 are used to prove the Theorem
 \ref{main}  in Section 4.

\section{Notations and preliminaries}
   
We will always denote points in $\mathbb{R}^{k} \times \mathbb{R}^{N - k}$ as pairs $x = (y, z)$, assuming $2 < k < N$ and 
$0 \leq t < 2$.\\
Through out this paper, we denote the norm of $H_{0}^{1}(\Omega)$ by $||u||= (\int_{\Omega} |\nabla u|^{2}dx)^{\frac{1}{2}}$, where 
$dx$ denote the Lebesgue measure in $\rn$.
 Let $u : \Omega \rightarrow \mathbb{R}$ be a real valued measurable function defined on $\Omega.$ We define 
\begin{equation}\label{n}
 |u|_{q,t,\Omega} = \left( \int_{\Omega}\frac{|u|^q}{|y|^t} dx \right )
^{\frac{1}{q}}
\end{equation}
and we say $u \in L^{q}_{t}(\Omega)$ if $ |u|_{q,t,\Omega}  < \infty.$

$D^{1,2}(\rn)$ will denote the closure of $C_{c}^{\infty}(\rn)$ with respect to the norm 
$||u||_{1} = (\int_{\rn}|\nabla u|^{2} dx)^\frac{1}{2}.$
We list here a few integral inequalities, for details we refer to \cite{V}. The first inequality we state is the Hardy inequality.\\

{\bf Hardy Inequality:} For $k > 2$ we have,  
\begin{equation}\label{E:2.1}
 C_{k}\int_{\rn} |y|^{-2} |u|^{2} dx \leq \int_{\rn} |\nabla u|^{2} dx , \; \forall u \in D^{1,2}(\rn),
\end{equation} \vspace*{.5 cm}

where $C_{k} = \left(\frac{k-2}{2}\right)^2$ is the best constant and is not attained.

For any $\lambda \in (0, (\frac{k-2}{2})^2),$ let us introduce the Hilbert space $H^{1}_{0}(\Omega)$ equipped 
with the inner product 
\[
 \langle u, v \rangle_{\lambda} : = \int_{\Omega} \left( \nabla u \nabla v - \frac{\lambda}{|y|^2} uv \right) dx,
\]
which induces the norm 
\begin{equation}\label{equi}
 ||u||_{\lambda} := \left[ \int_{\Omega} \left( |\nabla u|^2 - \lambda \frac{|u|^2}{|y|^2} \right) dx \right]^{\frac{1}{2}}.
\end{equation}
By the Hardy inequality \eqref{E:2.1}, we get that 
\begin{equation}\label{equi1}
 \left(1 - \frac{\lambda}{C_{k}} \right)^{\frac{1}{2}}||u|| \leq ||u||_{\lambda},
\end{equation}
hence $||.||_{\lambda}$ and $||.||$ are equivalent norms.


Another consequence of the Hardy inequality is that if  $\la < \frac{(k-2)^2}{4}$ then, 
\[
  L[.] \equiv - \De - \frac{\la}{|y|^2}  
\]
is positive definite and has discrete Spectrum in $ H^{1}_{0}(\Omega)$. 

 Let $\mu_{1}(\la)$ be the first eigenvalue of the operator $L[.]$ in $H^{1}_{0}(\Omega)$, then it is
  characterized by the following variational principle :
\begin{equation}\label{LO}
\mu_{1}(\la) = \min_{u \in H^{1}_{0}(\Omega)} \frac{\int_{\Omega} |\nabla u|^2 dx -
 \la \int_{\Omega} \frac{u^2}{|y|^2} dx}{\int_{\Omega} u^2} .
 \end{equation}\vspace*{.5 cm}

Now it is easy to note that, if $\mu \geq \mu_{1}(\la)$ , any nontrivial solution of \eqref{E:1.1} 
is sign changing. This can be seen by multiplying the first eigenfunction of 
the operator $(- \De - \frac{\la}{|y|^2})$ in $H^{1}_{0}(\Omega)$
with zero-boundary value problem and 
integrating both sides. Thus, by the result of \cite{WW}, Eq. \eqref{E:1.1} admits  
solutions and all solutions change sign.
 Hence, from now onwards we shall only consider $ 0 < \mu < \mu_{1}(\la)$.

The starting point for studying \eqref{E:1.1} is the Hardy-Sobolev-Maz'ya inequality, that is for the case $k < N$ and  was
proved by Maz'ya in \cite{V}. Now recall the HSM inequality.\\

{\bf Hardy-Sobolev-Maz'ya(HSM) Inequality:} Let $p > 2$ and $p \leq \frac{2N}{N-2}$ if $N \geq 3$. Let $t = N - \frac{N-2}{2}p$.
Then there is $C = C(N,p)$ such that 
\begin{equation}\label{E:2.2}
 \left(\int_{\mathbb{R}^k \times \mathbb{R}^{N-k}} \frac{|u|^p}{|y|^t} dy dz\right)^{\frac{2}{p}} \leq C 
\int_{\mathbb{R}^k \times \mathbb{R}^{N-k}} \left[|\nabla u|^2   - \frac{(k-2)^2}{4} \frac{u^2}{|y|^2} \right] dy dz
\end{equation}
for all $u \in C_{c}^{\infty}(\mathbb{R}^k \times \mathbb{R}^{N-k})$.\\

Let us derive the following weighted $L^{p}$ embedding.
\begin{lem}\label{LP}
 If $\Omega$ is a bounded subset of $\rn$ = $\mathbb{R}^k \times \mathbb{R}^{N-k}$, $0 \leq t < 2$, then
\[
 L^{p}_{t}(\Omega) \subset L^{q}_{t}(\Omega)
\]
with the inclusion being continuous, whenever $1 \leq q \leq p < \infty$.  
\end{lem}
\begin{proof}
 Let $1 \leq q < p < \infty$ and  $f \in L^{p}_{t}(\Omega)$. Then by H\"{o}lder's inequality we have 
\begin{eqnarray*}
 \int_{\Omega} \frac{|f|^{q}}{|y|^t} dx && = \int_{\Omega} \frac{|f|^{q}}{|y|^{t\frac{q}{p}}} \frac{1}{|y|^{t(1- \frac{q}{p})}} dx \\
&& \leq \left( \int_{\Omega} \frac{|f|^p}{|y|^t} dx  \right)^{\frac{q}{p}} \left( \int_{\Omega}\frac{1}{|y|^t} dx     
\right)^{\frac{p-q}{p}}\\
&& \leq C \left( \int_{\Omega} \frac{|f|^p}{|y|^t} dx  \right)^{\frac{q}{p}}.
\end{eqnarray*}
Since second term in the RHS is finite as $t < 2$ and $k > 2$, hence we have 
\[
 |f|_{q,t,\Omega} \leq C |f|_{p,t,\Omega}.
\]
This completes the proof.
\end{proof}

\begin{rem}
 If $f \in L^{p}_{t}(\Omega)$ for $1\leq p < \infty$, then clearly $f \in L^{p}(\Omega)$ with 
\[
 ||f||_{p} \leq C |f|_{p,t,\Omega}.
\]
\end{rem}\vspace*{.25cm}
 Let us prove the following compactness Result.
\begin{lem}\label{CP1}
 Let $1 \leq q < 2^{*}(t)$, $0\leq t < 2$, then the embedding $H^{1}_{0}(\Omega) \hookrightarrow L^{q}_{t}(\Omega)$ is compact.
\end{lem}
\begin{proof}
 Let $\{u_{n}\}_{n}$ be a bounded sequence in $H^{1}_{0}(\Omega)$. Then upto a subsequence we may assume $u_{n} \rightharpoonup u$ 
in $H^{1}_{0}(\Omega)$ and pointwise. To complete the proof we need to show  $u_{n} \rightarrow u$ in $L^{q}_{t}(\Omega)$. We estimate

\begin{eqnarray*}
 \int_{\Omega} \frac{|u_{n} -u|^q}{|y|^t} dx && = \iint_{|y|<\delta} \frac{|u_{n} -u|^q}{|y|^t} dy \ dz +
 \iint_{|y| \geq \delta} \frac{|u_{n} -u|^q}{|y|^t} dy \ dz\\
&& \leq \iint_{|y|<\delta} \frac{|u_{n} -u|^q}{|y|^t} dy \ dz +
 \frac{1}{\delta^t} \iint_{|y| \geq \delta} |u_{n} - u|^{q} dy \ dz.
\end{eqnarray*}
The convergence of the 2nd integral follow from Relich Compactness Theorem, since $2^{*}(t) < 2^*$. 
On the other hand, by H\"{o}lder's inequality
and Hardy-Sobolev-Maz'ya inequality \eqref{E:2.2}, we get, 
\begin{eqnarray*}
 \iint_{|y| < \delta} \frac{|u_{n} -u|^q}{|y|^t} dx  && \leq 
\left( \iint_{|y| < \delta}  \frac{dy  dz}{|y|^t}  \right)^{\frac{2^*(t) - q}{2^*(t)}} 
\left( \int_{\Omega} \frac{|u_{n} - u|^{2^*(t)}}{|y|^t} dx  \right)^{\frac{q}{2^*(t)}}\\
  && \leq C ||u_{n} -u||^{q}
\end{eqnarray*}
for some positive constant $C$. Therefore, for a given $\epsilon > 0$, we can choose  a $\delta$ such that 
\[
 \iint_{|y| < \delta} \frac{|u_{n} -u|^q}{|y|^t} dx < \frac{\epsilon}{2}.
\]
 Therefore, 
\[
 \int_{\Omega} \frac{|u_{n} -u|^q}{|y|^t} dx < \epsilon
\]
for large $n$. Hence this proves the lemma.
\end{proof}

 We recall that the solutions of \eqref{E:1.1} are the critical points of the energy functional given by   
    
\begin{equation}\label{LF}
J_{\la}(u) = \frac{1}{2} \int_{\Omega} |\nabla u|^{2} dx - \frac{\la}{2} \int_{\Omega} \frac{|u|^{2}}{|y|^{2}} 
- \frac{1}{2^{*}(t)} \int_{\Omega} \frac{|u|^{2^*(t)}}{|y|^{t}} dx - \frac{\mu}{2} \int_{\Omega} u^{2} dx.
\end{equation}
Then $J_{\la}$ is a well defined $C^{1}$ functional on $H^{1}_{0}(\Omega)$, thanks to the Hardy-Sobolev-Maz'ya inequality
\eqref{E:2.2}.  
    
 We use  variational methods in order to prove Theorem \eqref{main}. The main tool is an abstract theorem  
 proved by Schechter and Zou (see \cite{SZ}, Theorem 2). However we can not conclude theorem \eqref{main} by directly using the 
abstract theorem. This is because the variational problem corresponding to \eqref{E:1.1}
does not satisfy the Palais-Smale condition, therefore to overcome this difficulty we argue 
 as in \cite{SZ}, for each $\epsilon_{n} > 0 $ we 
 obtain a sequence $\{u_{l}^{n}\}_{l \in \mathbb{N}}$ of sign changing 
 solutions of the following subcritical problem 
 
 \begin{equation}\label{E:2.3}
\left.
 \begin{array}{rlll}
  -\De  u- \la \dfrac{u}{|y|^{2}} & = & \dfrac{|u|^{2^*(t)-2 - \epsilon_{n} }u}{|y|^{t}} + \mu u &{\rm in} \; \Omega, \\
u & = &  0 &{\rm on} \; \partial \Omega,
 \end{array} \right\}
\end{equation}
with a lower bound on the Morse index. Then we  prove that for fixed  $l \in \mathbb{N}$, 
 $\sup_{n \in \mathbb{N}} ||u^{n}_{l}||_{H^{1}_{0}(\Omega)}  
 < \infty$. These are dicussed in Section 3.

\section{Existence and Morse Index of sign-changing critical points}

In this section we will prove the existence of sign changing solutions for the perturbed compact problem \eqref{E:2.3} with 
an estimate on  Morse index. This is done by 
using the abstract theorem of Schechter and Zou (see \cite{SZ}, Theorem 2). However we cannot directly 
apply it due to the presence of 
the singular 
Hardy term  and Hardy-Sobolev-Maz'ya term in Eq. \eqref{E:2.3}, hence we need some precise estimates.\\

In the sequel we assume that  $H^{1}_{0}(\Omega)$ is endowed with $||.||_{\la}$  norm as defined in \eqref{equi} unless and 
otherwise mentioned.\\
   
 Let $0 < \mu_{1}(\la) < \mu_{2}(\la) \leq \mu_{3}(\la) \ldots \leq \mu_{l}(\la) \leq \ldots $ be the 
 eigenvalues of $(-\De - \frac{\la}{|y|^2})$
 on $H^{1}_{0}(\Omega)$
 and $\phi_{l}(x)$ be the eigenfunction corresponding 
 to $\mu_{l}(\la)$. Denote $E_{l} : = $span$\{\phi_{1}, \phi_{2}, \ldots, \phi_{l} \}$. Then 
 $H^{1}_{0}(\Omega)=\overline{ \cup_{l=1}^{\infty} E_{l}}$, dim$E_{l} = l$ and $E_{l} \subset E_{l+1}$.\\
We fix a $\epsilon_{0} > 0$ small enough and 
 choose a sequence $\epsilon_{n}$ in $(0, \epsilon_{0})$ such that $\epsilon_{n}\rightarrow 0$ in Eq. \eqref{E:2.3}.

\begin{thm}\label{T:3.1}

 Fix $\la \in [0, \frac{(k-2)^2}{4} -4)$, $\mu > 0$, then for every $n$ the Eq. \eqref{E:2.3} 
has infinitely many sign changing solutions 
$\{u^{n}_{l}\}_{l=1}^{\infty}$ such that for each $l$, the sequence $\{u_{l}^{n}\}_{n=1}^{\infty}$ is bounded 
in $H^{1}_{0}(\Omega)$ and the 
augmented Morse index of $u^{n}_{l}$ on the space $H^{1}_{0}(\Omega)$ is greater than or equal to $l$.
\end{thm}

Let us denote the energy functional corresponding to \eqref{E:2.3} by 
\begin{equation}\label{F}
 J_{\la, \epsilon_{n}}(u) = \frac{1}{2} \int_{\Omega} \left( |\nabla u|^2  - \frac{\la}{|y|^{2}} u^2 - \mu u^2 \right) dx
 - \frac{1}{2^*(t)- \epsilon_{n}}
\int_{\Omega} \frac{|u|^{2^*(t) - \epsilon_{n}}}{|y|^{t}} dx, 
\end{equation}
then the  singular term $\int_{\Omega} \frac{|u|^{2^*(t) -\epsilon }}{|y|^{t}}$is finite by Lemma \ref{LP} 
and Hardy-Sobolev-Maz'ya inequality \eqref{E:2.2}. 
 Hence $J_{\la,\epsilon_{n}}$ is a $C^{2}$, even functional on $H^{1}_{0}(\Omega)$. 
 In view of Lemma \ref{CP1}, $J_{\la,\epsilon_{n}}$ also 
 satisfies the Palais-Smale condition. In order to prove the Theorem \ref{T:3.1}, it is enough to obtain sign changing critical
points for the functional  $J_{\la, \epsilon_{n}}$ . 

 Recall the Augmented Morse index $m^* (u_{l}^{n})$ of $u^{n}_{l}$ in the space $H^{1}_{0}(\Omega)$ is defined as
 \[
  m^*(u_{l}^{n}) = \mbox{max}\{ \mbox{dim} H : H \subset H^{1}_{0}(\Omega)~ \mbox{is a subspace such that}~ 
  J_{\la, \epsilon_{n}}^{\prime\prime}(h,h) ~ \leq 0 ~ \forall ~ h ~ \in  H  \}.
 \]
For each $\epsilon_{n} \in (0, \epsilon_{0})$ fixed, we define,
\[
 ||u||_{n,*} := ||u||_{*} = \left[\int_{\Omega} 
\frac{|u|^{2^*(t) - \epsilon_{n}}}{|y|^{t}} dx    \right]^{\frac{1}{2^*(t)- \epsilon_{n}}} \; \; 
u \in H^{1}_{0}(\Omega),
\]
then from Lemma \ref{LP} and Hardy-Sobolev-Maz'ya inequality 
\eqref{E:2.2}, we get $||u||_{*} \leq C ||u||_{\la}$ for all $u \in H^{1}_{0}(\Omega)$ for some constant $C > 0$. Moreover we have,
for fixed $n,$ 
$||v_{k} - v||_{*} \rightarrow 0$ whenever $v_{k} \rightharpoonup v $ weakly in $H^{1}_{0}(\Omega)$, thanks to Lemma \ref{CP1}.

We write $P := \{u \in H^{1}_{0}(\Omega) : u \geq 0\}$ for the convex cone of nonnegative functions in $H^{1}_{0}(\Omega).$

Define for $\delta > 0$,
\[
 D(\delta) := \{ u \in H^{1}_{0} : \mbox{dist}(u, P) < \delta \}.
\]
  Denote the set of all critical points by
\[
 K^{\la}_{n} := \{ u \in H^{1}_{0}(\Omega) : J'_{\la, \epsilon_{n}}(u) = 0 \}.
\]

The important properties of $J_{\la, \epsilon_{n}}$ needed in the proof of Theorem \ref{T:3.1} are collected below.\\\\
Clearly $J_{\la, \epsilon_{n}}$ is a $ C^{2}$  even functional which maps bounded sets to bounded sets
 in terms of the norm $||.||_{\la}$. 
The gradient $J^{\prime}_{\la, \epsilon_{n}}$ is of the 
form $J_{\la, \epsilon_{n}}^{\prime}(u) = u - K_{\la, \epsilon_{n}}(u)$, 
 where $K_{\la, \epsilon_{n}} : H^{1}_{0}(\Omega) \rightarrow H^{1}_{0}(\Omega)$ is a continuous and compact operator. 
Now we are going to
study how the operator $K_{\la, \epsilon_{n}}$ behaves on $D(\delta)$. Let us prove the  following proposition.
\begin{prop}\label{P:3.2}
 For any $\rho_{0} > 0$ small enough, we have that $K_{\la, \epsilon_{n}}( D(\rho_{0})) \subset D(\rho) \subset D(\rho_{0})$ for
 some $\rho \in (0, \rho_{0}) $ for each $\la, n$ with $\la \in [0, \frac{(k-2)^2}{4} -4)$. Moreover, $D(\rho_{0}) 
\cap K^{\la}_{n} \subset P$.
\end{prop}
\begin{proof}
 First note that $K_{\la, \epsilon_{n}}(u)$ can be decomposed as $K_{\la, \epsilon_{n}}(u) = L(u) + W(u)$ where $L(u), W(u)$
$\in H^{1}_{0}(\Omega)$ are the unique solutions of the equations
\[
 - \De(L(u)) =  \mu u \; \mbox{and} \; -\De(W(u)) = \frac{|u|^{2^*(t) - \epsilon_{n} - 2} u}{|y|^{t}} .
\]
In other words, $L(u)$ and $W(u)$ are uniquely determine by the relations
\begin{equation}\label{PI}
\langle Lu, v\rangle_{\la} =   \mu \int_{\Omega} u v \ dx , 
 \;  \langle  W(u), v  \rangle_{\la} = \int_{\Omega} 
 \frac{|u|^{2^*(t) - \epsilon_{n} - 2} u v}{|y|^{t}} dx. 
\end{equation}
{\bf Claim} If $u \in P,$ then $L(u),$ $G(u) \in P.$\\
  {\bf Proof of claim} Define, $L(u)^{-} = \mbox{max} \{ 0, -L(u)\}.$ If $u \in P,$ we have $ u \geq 0,$ and it follows that
  \begin{align*}
   \langle  L(u), L(u)^{-}  \rangle_{\la} & = \int_{\Omega} 
\left( \nabla L(u) \nabla L(u)^{-} - \la \frac{L(u)L(u)^-}{|y|^2} \right) dx \\
   & = - \int_{\Omega} \left(  |\nabla L(u)^-|^2 - \la \frac{|L(u)^-|^2}{|y|^2} \right) dx \\
   & = \mu \int_{\Omega} u L(u)^{-} dx \geq 0.\\
 \end{align*}
Hence, we have 
\[
 \int_{\Omega} |\nabla L(u)^-|^2  dx \leq \la \int_{\Omega} \frac{|L(u)^-|^2}{|y|^2} dx,
\]
then by Hardy's inequality \eqref{E:2.1}, we get $L(u)^- = 0$ a.e., which imply $L(u) \in P.$ Similarly, we have $G(u) \in P.$ This proves 
the claim.

 Using relation \eqref{PI} we have,  
\begin{eqnarray*}
 \langle Lu, Lu  \rangle_{\la}  && = \mu \int_{\Omega} u Lu \ dx \\
&&  \leq \mu \left(\int_{\Omega} |u|^2 dx \right)^{\frac{1}{2}} 
\left(\int_{\Omega}  |Lu|^2  dx \right)^\frac{1}{2}\\
&& \leq  \frac{\mu}{\mu_{1}(\la)} \left( \int_{\Omega} \left [|\nabla u|^2 - \la \frac{|u|^2}{|y|^2} \right ] dx \right)^{\frac{1}{2}} 
\left( \int_{\Omega}  \left[ |Lu|^2  -  \la \frac{ |Lu|^2}{|y|^2} \right] dx  \right)^{\frac{1}{2}} \\
&& \leq \frac{\mu}{\mu_{1}(\la)} ||Lu||_{\la}   ||u||_{\la}.
\end{eqnarray*}

 Now since $\mu <   \mu_{1}(\la)$, we get, 
\[ 
 ||Lu||_{\la} \leq \alpha ||u||_{\la},
\]
where $\alpha < 1$. Let $u \in H^{1}_{0}(\Omega)$
 and $v \in P$ be such that dist$(u,P) = ||u-v||_{\la}$ , then
\begin{equation}\label{PI1}
 \mbox{dist}(Lu, P) \leq ||Lu - Lv||_{\la} \leq \alpha ||u-v||_{\la} \leq \alpha \mbox{dist}(u,P).
\end{equation}
 Next we shall estimate the distance between $W(u)$ and $P$. Set $u^{-}:= \max \{0, -u\}$, $p_{n}(t) = 2^*(t) - \epsilon_{n}$. Then
\begin{eqnarray*}
 && \mbox{dist}(W(u),P) ||W(u)^{-}||_{\la} \leq ||W(u)^{-}||^{2}_{\la} \leq \langle W(u), W(u)^{-}\rangle_{\la} \\
&& = \int_{\Omega} \frac{|u|^{2*(t) - \epsilon_{n} -2} u W(u)^{-}}{|y|^t} dx \leq
 \int_{\Omega} \frac{|u^{-}|^{2^*(t) - \epsilon_{n} -1}
|W(u)^{-}|}{|y|^t} dx \\
 && = \int_{\Omega} \frac{|u^-|^{p_{n}(t) -1}}{|y|^{t\frac{p_{n}(t)-1}{p_{n}(t)}}} 
\frac{|W(u)^{-}|}{|y|^{\frac{t}{p_{n}(t)}}} dx \\
&& \leq \left( \int_{\Omega} \frac{|u^-|^{p_{n}(t)}}{|y|^{t}} dx \right)^{\frac{p_{n}(t)-1}{p_{n}(t)}}
\left( \int_{\Omega} \frac{|W(u)^-|^{p_{n}(t)}}{|y|^{t}} dx \right)^{\frac{1}{p_{n}(t)}},
\end{eqnarray*}
 the second term of the right hand side could be estimated by using Lemma \ref{LP} and Hardy-Sobolev-Maz'ya inequality 
\eqref{E:2.2}, hence we get,
\[
  \left( \int_{\Omega} \frac{|W(u)^-|^{p_{n}(t)}}{|y|^{t}} dx \right)^{\frac{1}{p_{n}(t)}} \leq  C ||W(u)^{-}||_{\la}.
\]
 Using Lemma \ref{LP} and inequality \eqref{E:2.2}, we obtain,  
\[
 \int_{\Omega} \frac{|u^{-}|^{p_{n}(t)}}{|y|^{t}} dx = \min\limits_{v \in P} \int_{\Omega} \frac{|u-v|^{p_{n}(t)}}{|y|^{t}} \leq
C \min \limits_{v \in P} ||u-v||_{\la}^{p_{n}(t)}.
\]
Thus we get,
\[
 \mbox{dist} (W(u), P) \leq C [\mbox{dist}(u, P)]^{p_{n}(t)-1} \quad \forall u \in H^{1}_{0}(\Omega).
\]
Choose $ \alpha < \nu < 1$, then there exists $\rho_{0}$ such that, if $\rho \leq
\rho_{0}$
 \begin{equation}\label{PIF}
 \mbox{dist}(W(u),P) \leq (\nu - \alpha ) \mbox{dist}(u,P) \quad \mbox{for
all}\quad  u\in D(\rho).
\end{equation}                                            
 Fix $\rho \leq \rho_{0}$. Inequalities \eqref{PI1} and \eqref{PIF} yield 
\begin{eqnarray*}
 \mbox{dist}(K_{\la, \epsilon_{n}}(u),P) && \leq \mbox{dist}(L(u),P) + \mbox{dist} (W(u),P)\\
&& \leq \nu \mbox{dist} (u,P)
\end{eqnarray*}
for all $u\in D(\rho)$, hence we have $K_{\la, \epsilon_{n}}( D(\rho_{0})) \subset D(\rho) \subset D(\rho_{0})$ for
 some $\rho \in (0, \rho_{0}).$ Also from this relation we deduce that, for any $\rho < \rho_{0}$ it holds, 
\begin{equation}
 u \in \partial D(\rho) \Longrightarrow ||J^{\prime}_{\lambda,\epsilon_{n}}(u)|| \geq C,
\end{equation}
where $C > 0$ is a positive constant. This in particular means that if $u \in D(\rho_{0}) \cap K_{n}^{\lambda}$ then 
indeed $u \in P.$
 This proves  the Proposition.                                                      
\end{proof}

Now we want to examine how the functional $J_{\la, \epsilon_{n}}$ behaves on each finite dimensional space $E_{l}$.

\begin{prop}\label{P:3.3}
 For each l, $\lim_{||u|| \rightarrow \infty, u \in E_{l}} J_{\la, \epsilon_{n}}(u) = - \infty$
\end{prop}
\begin{proof}
 For each $n$ , $||.||_{*}$ defines a norm on $H^{1}_{0}(\Omega)$, now since $E_{k}$ is  finite dimensional, there exists 
 a constant $C > 0$ such that $||u||_{\la} \leq C ||u||_{*}$ for all $u \in E_{k}$. Thus 
\begin{eqnarray*}
 J_{\la, \epsilon_{n}}(u) && \leq  \frac{1}{2} \int_{\Omega} |\nabla u|^{2} dx - \frac{\la}{2} \int_{\Omega} \frac{|u|^2}{|y|^2} dx - 
C \int_{\Omega} \frac{|u|^{2^{*}(t) -\epsilon_{n}}}{|y|^{t}} dx \\
&& \leq \frac{1}{2} ||u||_{\la}^{2} - C ||u||_{\la}^{2^{*}(t) -\epsilon_{n}}. 
\end{eqnarray*}
 Since $2^{*}(t) - \epsilon_{n} > 2$, we have $\lim_{||u|| \rightarrow \infty, u \in E_{k}} J_{\la, \epsilon_{n}}(u) = - \infty$.
\end{proof}
\begin{prop}\label{P:3.4}
 For any $\alpha_{1}, \alpha_{2} > 0$, there exists an $\alpha_{3}$ depending on $\alpha_{1}$ and $\alpha_{2}$ such that 
$||u|| \leq \alpha_{3}$ for all $u \in J_{\la, \epsilon_{n}}^{\alpha_{1}} \cap\{ u \in H^{1}_{0}(\Omega) :
||u||_{*} \leq \alpha_{2} \}$ where $J^{\alpha_{1}}_{\la, \epsilon_{n}} = \{u \in H^{1}_{0}(\Omega) : J_{\la, \epsilon_{n}}(u)
\leq \alpha_{1}\}$. 
\end{prop}
\begin{proof}
 Using Hardy's inequality \ref{E:2.1}, we have, 
\begin{eqnarray*}
 J_{\la, \epsilon_{n}}(u) + \frac{1}{2^*(t) - \epsilon_{n}} ||u||_{*}^{2^*(t)- \epsilon_{n}}&& = 
\frac{1}{2} \int_{\Omega} |\nabla u|^{2} dx -
\frac{\mu}{2} \int_{\Omega} u^{2} dx - \frac{\la}{2} \int_{\Omega} \frac{|u|^2}{|y|^{2}} dx \\
&& = \frac{1}{2}||u||_{\la}^2 - \frac{\mu}{2} \int_{\Omega} |u|^2 dx \\
&&  \geq  \frac{1}{2} \left[||u||_{\la}^{2} - \frac{\mu}{\mu_{1}(\la)} ||u||_{\la}^{2}    \right].
\end{eqnarray*}
Thus we have
$J_{\la, \epsilon_{n}}(u) + \frac{1}{2^*(t) - \epsilon_{n}} ||u||_{*}^{2^*(t)- \epsilon_{n}} \geq 
\frac{1}{2} \left[ \frac{\mu_{1}(\la) - \mu}{ \mu_{1}(\la)}\right]  ||u||_{\la}^2$. \\
Hence the proposition follows.
\end{proof}
Now we can prove Theorem \ref{T:3.1}.\\
{\bf Proof of Theorem \ref{T:3.1}} The propositions \ref{P:3.2}, \ref{P:3.3} and \ref{P:3.4} tell us that $J_{\la, \epsilon_{n}}$
satisfies all the conditions of Theorem 2 in \cite{SZ}. Thus $J_{\la, \epsilon_{n}}$ has a sign changing critical point 
$u^{n}_{l} \in H^{1}_{0}(\Omega)$ at  a level $C(n, \la, l)$ where  $C(n,\la, l) \leq \mbox{sup}_{E_{l+1}} J_{\la, \epsilon_{n}}$ 
and the augmented Morse index $m^*(u^n_{l})$ of $u^{n}_{l}$ is greater than or equal to  $l$. The only 
things remains to show is that the 
sequence $\{u^n_{l}\}_{n=1}^{\infty}$ is bounded for each $l$. \\
{\bf Claim} There exists a constant $T_{1} > 0$ independent of $l$ and $n$ such that 
\[
 \sup_{E_{l+1}} J_{\la, \epsilon_{n}}(u) \leq T_{1}\mu_{l+1}(\la)^{\frac{2^*(t) - \epsilon_{0}}{2(2^*(t)- \epsilon_{0}-2)}}.
\]
{\bf Proof of Claim} The definition of $E_{l +1}$ implies that $||u||_{\la}^{2} \leq \la_{l+1}||u||^2_{L^2} \leq C \la_{l+1}
|u|^{2}_{p,t,\Omega}$. Note we have $|u|_{2^*(t) - \epsilon_{0},\Omega} \leq D_{1} |u|_{2^*(t) - \epsilon_{n},\Omega}$,
where $D_{1} > 0$ is a constant, independent of $n$ and $k$. Thus we have 
\begin{eqnarray*}
 J_{\la, \epsilon_{n}}(u)&&  \leq \frac{1}{2} \int_{\Omega} |\nabla u|^{2} dx - \frac{\la}{2} \int_{\Omega} \frac{|u|^2}{|y|^2} dx -
\frac{1}{2^*(t) - \epsilon_{n}} \int_{\Omega} \frac{|u|^{2^*(t)- \epsilon_{n}}}{|y|^t} dx \\
 && =  \frac{1}{2}||u||_{\la}^{2} - \frac{1}{2^*(t) - \epsilon_{n}} \int_{\Omega} \frac{|u|^{2^*(t)- \epsilon_{n}}}{|y|^t} dx.
\end{eqnarray*}
 Now using the inequality 
\[
 |u|^{2^*(t) - \epsilon_{0}} \leq c_{1} |u|^{2^*(t) - \epsilon_{n}} + c_{2},
\]
and the fact $t < 2$, we get, 
\[
 J_{\la, \epsilon_{n}}(u) \leq \frac{1}{2} ||u||_{\la}^{2} - D_{2}
 \int_{\Omega} \frac{|u|^{2^*(t)- \epsilon_{0}}}{|y|^t} dx + D_{3},
\]
where $D_{2} > 0$, $D_{3} > 0$ are constants independent of $n$ and $l$. Also there exists a constant $D_{4} >0$ such that 
$|u|_{2,t,\Omega} \leq D_{4} |u|_{2^*(t) -\epsilon_{0}, t, \Omega}$, therefore we may have $D_{5}$ such that 
$||u||^{2^*(t) - \epsilon_{0}}_{\la} \leq D_{5} \la_{l+1}^{({2^*(t)- \epsilon_{0})}/2} 
|u|_{2^*(t)-\epsilon_{0}, t, \Omega}^{2^*(t)- \epsilon_{0}}$ for all $u \in E_{l+1}$. Then
\begin{eqnarray*}
 J_{\la, \epsilon_{n}}(u)  && \leq \frac{1}{2} ||u||^{2}_{\la} - D_{6} \mu_{l+1}(\la){^{-(2^*(t)-\epsilon_{0})/2} ||u||_{\la}^{2^*(t) - \epsilon_{0}}} 
+ D_{3}\\
&& \leq D_{7} \mu_{l+1}(\la)^{{\frac{2^*(t)- \epsilon_{0}}{2(2^*(t)-\epsilon_{0} -2)}}} + D_{3}\\
&& \leq T_{1} \mu_{l+1}(\la)^{{\frac{2^*(t)- \epsilon_{0}}{2(2^*(t)-\epsilon_{0} -2)}}},
\end{eqnarray*}
where $D_{i}(i= 1, \ldots,7)$ and $T_{1}$ are positive constants independent of $l$ and $n$.\\
Also we know that energy of any critical point of $J_{\la, \epsilon_{n}}$ is 
nonnegative. Thus $J_{\la, \epsilon_{n}}(u^{n}_{l}) \\
\in [0, T_{1}\mu_{l+1}(\la){^{\frac{2^*(t)- 
\epsilon_{0}}{2(2^*(t) - \epsilon_{0} -2)}}]}$. This subsequently implies that the sequence $\{u^n_{l}\}_{n=1}^{\infty}$ is 
bounded in $H^{1}_{0}(\Omega)$ for each $l$. This proves the Theorem.  
                           
 \section{Proof of the Main Theorems}                         
 
 In this section we will prove the existence of infinitely many solutions for the Eq. \eqref{E:1.1}. First we recall the following
  compactness results by C.Wang and J.Wang (see \cite{WW}, Theorem 1.3 for a proof):
 
\begin{lem}\label{COMPACT}
 Let $\mu > 0$, $N > 6 + t$, $\la \in [0, \frac{(k-2)^2}{4} -4)$ and $u_{n}$ is a  solution of \eqref{E:2.3} for each $\epsilon_{n}$.
 Suppose $\epsilon_{n} \rightarrow 0$ and  $||u_{n}|| \leq C$ for 
some constant independent of $n$, then $u_{n}$ has a subsequence, 
 which converges strongly in $H^{1}_{0}(\Omega)$ as $n \rightarrow + \infty$. 
\end{lem}

 {\bf Proof of Theorem \ref{main}}: The proof is divided into two steps.
 
 Step-1: By combining  Lemma \ref{COMPACT} and Theorem \ref{T:3.1}, we get a sequence $\{u_{l}\}_{l=1}^{\infty}$ of solutions of the 
 problem  \eqref{E:1.1}  with energy $C(\la, l) \in [0, T_{1} 
\mu_{l+1}(\la)^{{\frac{2^*(t) -\epsilon_{0}}{2(2^*(t)-\epsilon_{0} -2)}}}]$. Moreover, we claim that $u_{l}$ is still sign changing.
Since $\{u^n_{l}\}_{n=1}^{\infty}$ is a sign changing solutions to \eqref{E:2.3} , let
\[
 (u^n_{l})^{\pm} := \mbox{max} \{\pm u^n_{l}, 0 \}.
\]
Then we have 
\[
 \int_{\Omega} |\nabla (u^n_{l})^\pm|^2 dx = \la \int_{\Omega} \frac{|(u^n_{l})^\pm|^2}{|y|^2} dx 
+ \mu \int_{\Omega} |(u^n_{l})^\pm|^2 dx 
+ \int_{\Omega}\frac{|(u^n_{l})^\pm|^{2^*(t) -\epsilon_{n}}}{|y|^{t}} dx.
\]
Thus
\[
 ||(u^n_{l})^\pm||_{\la}^2 \leq \alpha ||(u^n_{l})^\pm||_{\la}^2 + \int_{\Omega}\frac{|(u^n_{l})^\pm|^{2^*(t) -\epsilon_{n}}}{|y|^{t}},
\]
where $\alpha < 1$. Then using Hardy-Sobolev-Maz'ya inequality \eqref{E:2.2} it follows that 
\[
 ||(u^n_{l})^\pm||_{\la} \geq C_{0} > 0.
\]
where $C_{0}$ is a constant independent of $n$. This implies that the limit $u_{l}$ of the subsequence $\{u^n_{l}\}$ is still 
sign-changing. 

Step-2 :  Now it remains to show that infinitely many $u_{l}'s$ are different. This follows if we show that energy of $u_{l}$ 
goes to infinity as $l \rightarrow \infty$. \\
 Suppose not,
then $\lim_{l\rightarrow \infty}C(\la,l) = c'< \infty$.
For any $l\in \mathbb{N}$ we may find an $n_{l}$(assume $n_{l} > l$) such that
$|C(n_{l},\la,l) - C(\la,l)| < \frac{1}{l}$. It follows
 that $\lim_{l\rightarrow \infty} C(n_{l},\la, l) = \lim_{l\rightarrow \infty}
C(\la,l) = c'<\infty$. Hence, 
 $\{u^{n_{l}}_l\}_{l\in \mathbb{N}}$ is bounded in $H^{1}_{0}(\Omega)$ and hence satisfies the uniform bound
given by Lemma \ref{COMPACT}. Therefore the augmented Morse index of $u^{n_l}_l$ remains bounded which contradicts the fact that 
the augmented Morse index of $u^{n_l}_l$ is greater than or equal to $l$. Thus $\lim_{l \rightarrow \infty} C(\la,l)
= \infty$ and hence infinitely many $u_{l}'s$ are different. This finishes the proof of Theorem \ref{main}.\\

{\bf Proof of Theorem \ref{non}}
  The proof  is based on the 
 Pohozaev identity. The difficulty in applying this identity is because of the presence singular terms. We can overcome this difficulty
 by using the partial $H^2$- regularity.
 With an obvious modification from ( \cite{SB}, Theorem 2.4), we can prove if $u$ is a solution 
 to the Eq. \eqref{E:1.1}, then $u_{z_{i}} \in H^{1}(\Omega)$ for all $1 \leq i \leq N-k$.\\
   To make the test function smooth we introduce cut-off functions and pass to the limit with the 
  help of the above regularity result. 
We will assume without loss of generality that $\Omega$ is star shaped with respect to the origin.
  
  For $\epsilon > 0$ and $R > 0$, define $\varphi_{\epsilon, R} = \varphi_{\epsilon}(x) \psi_{R}(x)$ where $\varphi_{\epsilon}(x) = 
  \phi(|y|/\epsilon)$, $\psi_{R}  = \psi(|x|/R)$, $\varphi$ and $\psi$ are smooth functions in $\mathbb{R}$ with the properties 
  $0 \leq \varphi, \psi \leq 1$, with supports of $\varphi$ and $\psi$ in $(1, \infty)$ and $(-\infty, 2)$ respectively and 
  $\varphi(t) = 1$ for $t \geq 2$, and $\psi(t) = 1$ for $t \leq 1$.

  Assume that \eqref{E:1.1} has a nontrivial solution $u$. Then $u$ is smooth away from the singular set and hence 
  $(x. \nabla u)\varphi_{\epsilon, R} \in \mbox{C}^{2}_{c} (\Omega)$. Multiplying Eq.\eqref{E:1.1} by this test function 
  and integrating by parts, we have 
  \begin{eqnarray}\label{NON}
  && \int_{\Omega} \nabla u. \nabla ((x. \nabla u) \varphi_{\epsilon, R}) \ dx - \la \int_{\Omega} 
  \frac{u (x. \nabla u)\varphi_{\epsilon,R}}{|y|^2} dx  - \int_{\Omega} \frac{\partial u}{\partial \nu} 
  (x . \nabla u)\varphi_{\epsilon , R} \ dx \notag \\
   &&  = \int_{\Omega} \frac{|u|^{2^{*}(t) -2} u}{|y|^t} (x. \nabla u) \varphi_{\epsilon, R} \ dx + 
  \mu \int_{\Omega} u (x. \nabla u)\varphi_{\epsilon , R} \ dx.
  \end{eqnarray}
Now, RHS of \eqref{NON} can be simplified as 
\begin{eqnarray*}
 && \int_{\Omega} \frac{|u|^{2^{*}(t) -2} u}{|y|^t} (x. \nabla u) \varphi_{\epsilon, R} \ dx + 
  \mu \int_{\Omega} u (x. \nabla u)\varphi_{\epsilon , R} \ dx \notag \\
  && = \frac{1}{2^{*}(t)} \int_{\Omega} (\nabla |u|^{2^{*}(t)}. x)\frac{\varphi_{\epsilon, R}}{|y|^t} dx  + \frac{\mu}{2} dx
  \int_{\Omega} (\nabla |u|^2. x) \varphi_{\epsilon, R} \ dx\\
  && = -\frac{(N-2)}{2} \int_{\Omega}\frac{|u|^{2^{*}(t)}}{|y|^t} \varphi_{\epsilon, R} \ dx - 
  \frac{1}{2^*(t)} \int_{\Omega} \frac{|u|^{2^*(t)}}{|y|^t} [x . (\psi_{R} \nabla \varphi_{\epsilon} dx + 
  \varphi_{\epsilon} \nabla \psi_{R})] \ dx\\
  && - \frac{N\mu}{2} \int_{\Omega} |u|^2 \varphi_{\epsilon, R} \ dx - \frac{\mu}{2} 
  \int_{\Omega} |u|^2 [x . (\psi_{R} \nabla \varphi_{\epsilon}  dx + \varphi_{\epsilon} \nabla \psi_{R})] \ dx.
  \end{eqnarray*}
Note that $|x . (\psi_{R} \nabla \varphi_{\epsilon} + \varphi_{\epsilon} \nabla \psi_{R})| \leq C$ and hence using the dominated 
convergence theorem we get
\begin{equation}\label{D}
 \lim_{R \rightarrow \infty} [\lim_{\epsilon \rightarrow 0} RHS] = - \frac{(N-2)}{2} \int_{\Omega} \frac{|u|^{2^*(t)}}{|y|^t} dx
 - \frac{N\mu}{2} \int_{\Omega} |u|^2 \ dx.
\end{equation}
For LHS, using integration by parts and the fact that $u_{z_{i}} \in H^{1}(\Omega)$, (see \cite{SB}, Theorem 4.1 for detail), we get
\begin{equation}\label{E}
 \lim_{R \rightarrow \infty} [\lim_{\epsilon \rightarrow 0} LHS]  = - \frac{(N-2)}{2} 
 \left[\int_{\Omega} \left( |\nabla u|^2 - \la \frac{u^2}{|y|^2} \right) dx \right] 
 - \frac{1}{2} \int_{\partial \Omega} \left(\frac{\partial u}{\partial \nu}   \right)^2 (x. \nu) \  dx.
\end{equation}
Substituting \eqref{D} and \eqref{E} in \eqref{NON}, and using Eq. \eqref{E:1.1}, we get 
\[
 \int_{\partial \Omega}\left(\frac{\partial u}{\partial \nu}\right)^2 (x. \nu)\  dx  + 2 |\mu| \int_{\Omega} u^2 \ dx = 0
\]
which implies $u =0$ in $\Omega$ by the principle of unique continuation. This proves the theorem.

\end{document}